\documentclass{article}  %

 \usepackage{amsmath}    

 \usepackage{amsthm}     

 \usepackage{amscd}      

 \usepackage{amssymb}    

 \usepackage{eucal}      

 \usepackage{latexsym}   

 \usepackage{graphicx}   

 \usepackage{verbatim}   

\newcommand{\eps}{\varepsilon}
\newcommand{\R}{\mathbb{R}}

\newtheorem{theorem}{Theorem}[section]

\newtheorem{assumption}[theorem]{Assumption}

\newtheorem{proposition}[theorem]{Proposition}

\newtheorem{remark}[theorem]{Remark}

\newtheorem{lemma}[theorem]{Lemma}

\begin{document}

\title{Weak in Space, Log in Time Improvement of the Lady{\v{z}}enskaja-Prodi-Serrin Criteria}

\author{Clayton Bjorland\ \ \ \ \ \ \ \ \ \ \ Alexis Vasseur\\ University of Texas, Austin}

\maketitle

\section{Introduction}

\subsection{Preliminary}

We consider solutions of the Navier-Stokes equation in $\mathbb{R}^3$ given by
\begin{align}
\partial_t u+ \nabla\cdot (u\otimes u) +\nabla P -\triangle u &= 0\label{NSE}\\
\nabla \cdot u &= 0.\notag
\end{align}
This equation is complemented by an initial condition $u(0)=u_0$ and some condition for the behavior as $|x|\rightarrow \infty$, typically $|u|\rightarrow 0$ which is made precise by considering solutions in certain functions spaces.  The literature regarding solutions for these equations is quite large and we discuss only a small subset which is immediately relevant to this paper.  A general open question for solutions is to discover conditions which guarantee a solution is ``regular'' (or smooth) for all time.  For example, given an initial condition $u_0\in L^2(\mathbb{R}^3)$ Leray \cite{Leray} proved there exists a solution (typically called a Leray-Hopf solution, see also \cite{MR0050423}) $u\in L^\infty((0,\infty);L^2(\mathbb{R}^3))$ and $\nabla u\in L^2((0,\infty);L^2(\mathbb{R}^3))$.  If $u_0$ is regular enough,
the solution immediately enters the class of $C^\infty$ smooth functions and remains there for some possibly small time (i.e. $u\in C^\infty((0,T^*)\times\mathbb{R}^3$) but it remains an open question weather it retains this smoothness property for all time (i.e. $u\in C^\infty((0,\infty)\times\mathbb{R}^3$).

Further work by Lady{\v{z}}enskaja \cite{MR0236541}, Prodi \cite{MR0126088} and Serrin \cite{MR0150444} established criteria for regularity which states that if a Leray-Hopf solution satisfies $u\in L^p((0,\infty);L^q(\mathbb{R}^3))$ with $p$ and $q$ satisfying the relation $\frac{2}{p}+\frac{3}{q}=1$, $2\leq p<\infty$ then $u$ is regular on $(0,\infty)\times \mathbb{R}^3$.   Note that all those spaces are invariant through the universal scaling of the Navier-Stokes equation:
$$
u_\eps(t,x)=\eps u(t_0+\eps^2 t, x_0+\eps x).
$$
Since $p<\infty$, Lebesgue's theorem ensures that their norms actually locally shrink when $\eps$ goes to 0.
That is
$$
\lim_{\eps\to0}\|u_\eps\|_{L^p(0,T;L^q(B))}=\lim_{\eps\to0}\|u\|_{L^p(t_0,t_0+\eps^2T;L^q(x_0+\eps ))}=0,
$$
for any ball $B\subset \R^3$.
The cases of invariant norms which do not locally shrink are far more difficult. For $p=\infty$, $L^\infty((0,T);L^3(\mathbb{R}^3))$ is such a space.  The criteria was expanded into this case by Iskauriaza, Ser\"egin and Shverak in \cite{MR1992563}.  

Another example of  spaces which do not shrink locally are the Lorentz spaces $L^{p,r}((0,T);L^{q,\infty}(\mathbb{R}^3))$ with $r\in [p,\infty]$ and $p$, $q$ satisifying the relation $\frac{2}{p}+\frac{3}{q}=1$, $2< p<\infty$.  Here $L^{p,p}$ is the standard Lebesque space, $L^{p,\infty}$ is the weak-$L^p$ space (defined below), and $L^{p,r}$ are the interpolation spaces.  Sohr, \cite{MR1877269}, considered these spaces and proved regularity if a Leray-Hopf solution is bounded in $L^{p,r}((0,T);L^{q,\infty}(\mathbb{R}^3))$ for $r\in (p,\infty)$.  The case $r=\infty$ is also considered and regularity is proved if the $L^{p,\infty}((0,T);L^{q,\infty}(\mathbb{R}^3))$ norm is small.  Additionally, Kozono and Yamazaki, \cite{MR1368344}, and Kozono, \cite{MR1808641}, prove existence and give regularity criteria when initial data is given in the space $L^{d,\infty}$ where $d$ is the dimension of the space in which the fluid evolves.

Recently the Lady{\v{z}}enskaja-Prodi-Serrin criteria has been extended to include  ``log improvements''.
Note that this is a family of spaces which are not invariant anymore through the universal scaling. Indeed, the local norm may blow up through the scaling when $\eps \to 0$. In this sense they are ``subcritical.''  A first log improvement in time only was proposed by  
Montgomery-Smith \cite{MR2160072}, using a Gronwall's argument.  This result was extended in time and space by Chan and Vasseur \cite{MR2437103}, using blow-up methods and De Giorgi techniques. Their proof was simplified and extended by  Zhou and Lei, \cite{zhoulei} using energy estimates.  The result by Zhou and Lei states that regularity is retained at time $T$ if
\begin{align}
\int_0^T\frac{\|u(t)\|^p_{L^q(\mathbb{R}^3)}}{1+\log(e+\|u(t)\|_{L^\infty(\mathbb{R}^3)})}\,dt <\infty\notag
\end{align}
where $p$ and $q$ satisfy the Prodi-Serrin condition $\frac{2}{p}+\frac{3}{q}=1$.

The goal of this paper is to further generalize the log improved Lady{\v{z}}enskaja-Prodi-Serrin criteria to weak-$L^p$ spaces defined as follows:
Let $\mu$ denote the Lebesgue measure on $\mathbb{R}^n$ and $\Omega\subset \mathbb{R}^n$.  Then $f\in L^{p,\infty}(\Omega)$ if and only if $\|f\|_{L^{p,\infty}(\Omega)}<\infty$ where
\begin{align}
\|f\|_{L^{p,\infty}(\Omega)}&=\sup_{\alpha>0}\{\alpha  [\lambda_{f,\Omega}(\alpha)]^{\frac{1}{p}}\}\notag\\
\lambda_{f,\Omega}(\alpha)&= \mu\{x\in \Omega| |f(x)|>\alpha\}.\notag
\end{align}

The main result of the paper is the following:
\begin{theorem}
Let $u$ be a Leray-Hopf solution of (\ref{NSE}),  defined on $(0,T)\times \mathbb{R}^3$, and satisfying the generalized energy inequality (\ref{genergy}).  If the solution also satisfies the bound
\begin{align}
\int_0^T\frac{\|u(s)\|^p_{L^{q,\infty}(\mathbb{R}^3)}}{e+\log(e+\|u(s)\|_{L^\infty(\mathbb{R}^3)})}\,ds < \infty\notag
\end{align}
then $\|u\|_{L^\infty(\mathbb{R}^3\times (\lambda,T])}<\infty$ for any $\lambda>0$.
\end{theorem}
\begin{remark}
We remark that the bound
\begin{align}
\int_0^T\left\|\frac{u(s)}{e+\log(e+|u(x,s)|)}\right\|^p_{L^{q,\infty}(\mathbb{R}^3)}\,ds < \infty\notag
\end{align}
implies the assumed bound in the above theorem.
\end{remark}

It is well known (since the work of Leray) that the $L^\infty$ bound above implies regularity.  
A current open problem is to deduce regularity at $(t_0,x_0)$ for solutions of the Navier-Stokes equation assumed to be bounded by 
\begin{align}
f(t,x)= \frac{C}{\sqrt{(x-x_0)^2+(t_0-t)}}, \qquad C>0.\notag
\end{align}
Such bounds do not a priori eliminate singularities at $(x_0,t_0)$.  Results have been obtained in the case of axisymmetric solutions by Seregin and {\v{S}}ver\'{a}k \cite{MR2512858}, and also by Chen, Strain, Tsai, and Yau  \cite{MR2512859}, \cite{MR2429247}. However, the problem remains open in the general setting.   It is easily checked through direct calculation that $f\in L^{p,\infty}((0,t_0);L^{q,\infty})(\mathbb{R}^3)$ with $q$ and $p$ satisfying the usual Prodi-Serrin relation $\frac{2}{p}+\frac{3}{q}=1$, which justifies the interest of the weak $L^p$ space.
Note that our theorem ensures the smoothness of any solution of Navier-Stokes equation bounded by a function
$$
f(t,x)=\frac{a(t)}{\sqrt{(x-x_0)^2+(t_0-t)}},
$$
if 
$$
\int_0^{t_0} \frac{a(t)^p}{(t-t_0)(e+\log [e+a(t)/\sqrt{t_0-t}])}\,dt<\infty.
$$
This does not include constant $C$.  In the appendix we give an example of a function $a$ for which the above bound is true but the function $f$ is not in $L^{p,r}((0,T);L^{q,\infty}(\mathbb{R}^3))$ for any $r\in(0,\infty)$. 

In the context of weak $L^{p,\infty}$ spaces, the energy method of Zhou and Lei collapses. Our approach is inherited from the Chan Vasseur paper. It is based on blow-up techniques and the application of the De Giorgi method to the Navier-Stokes equation developed by Vasseur \cite{MR2374209}. Note that this kind of  De Giorgi technique was previously used by Beir\~{a}o da Veiga  for the Navier-Stokes equation in other contexts \cite{MR1608053}, \cite{MR1648218}. The method requires the solution to satisfy the following generalized energy inequality in the sense of distributions:
\begin{align}
\partial_t \frac{|u|^2}{2}+\nabla \cdot (u\frac{|u|^2}{2})+\nabla\cdot (uP)+|\nabla u|^2-\triangle \frac{|u|^2}{2}\leq 0.\label{genergy}
\end{align}
This generalized energy inequality was introduced in partial regularity work of Caffarelli, Kohn and Nirenberg \cite{MR673830}.

The paper is organized as follows.  The remainder of this section contains useful lemmas involving weak $L^p$ spaces as well as the local energy flux estimates.  The second section is devoted to establishing an inequality which will be used in the final section to establish the regularity result mentioned above.  This inequality is established by bounding the initial energy locally with weak $L^p$ norms which we can scale small.  Combining these bounds with a recursive lemma and a scaling argument it is shown, using the flux estimates presented at the end of this section,  that the $L^\infty$ norm of the solution is bounded by certain weak $L^p$ norms.  Section three combines the inequality established in section two with a Gronwall argument to prove the main theorem mentioned above.  






\subsection{Weak $L^p$ Spaces}






We now summarize some well known properties of the weak spaces providing elementary proofs where appropriate.  The following lemma and corollary gives bounds for functions on compact sets.  Throughout, the indicator function of the set $E$ will be denoted $\chi_E$ (i.e $\chi_E(x)=1$ if $x\in E$ and $\chi_E(x)=0$ if $x\notin E$).

\begin{lemma}\label{weaklemma}
Let $K\subset \mathbb{R}^n$ be a compact set.  If $p<r$ the following hold
\begin{itemize}
 \item $\|f\|_{L^{p,\infty}(K)}\leq C(K)\|f\|_{L^{r,\infty}(K)}$
 \item For any $\epsilon>0$ there exists a $C(\epsilon,K)$ such that
\begin{align}
\|f\|_{L^p(K)}\leq C(\epsilon) \|f\|^{\frac{r}{p}}_{L^{r,\infty}(K)}+\epsilon\mu(K)^{\frac{1}{p}}.\notag
 \end{align}
\end{itemize}
\end{lemma}
\begin{proof}
The first statement follows from the similar property for $L^p$ spaces.
\begin{align}
\lambda_{f,K}(\alpha)^{\frac{1}{p}} &= \left( \int_K \chi^p_{|f|>\alpha} \,d\mu\right)^{\frac{1}{p}} \leq C(K)\left( \int_K \chi^r_{|f|>\alpha} \,d\mu\right)^{\frac{1}{r}}
\leq C(K)\lambda_{f,K}(\alpha)^{\frac{1}{r}}. \notag
\end{align}
To prove the second statement we use a well known description of $L^p$ norms.
\begin{align}
\int_K |f|^p \,d\mu &= p\int_0^\infty \alpha^{p-1}\lambda_f(\alpha)\,d\alpha\notag\\
&\leq p\int_\epsilon^\infty\alpha^{p-r-1}\{\alpha^r\lambda_f(\alpha)\}\,d\alpha + \epsilon^p\mu(K)\notag\\
&\leq p\|f\|^r_{L^{r,\infty}(K)}\left(\int_\epsilon^\infty\alpha^{p-r-1}\,d\alpha \right)+ \epsilon^p\mu(K).\notag
\end{align}
\end{proof}
















We now recall a useful way to move from $L^p$ to weak $L^p$ spaces that will help when bounding the pressure terms through the Riesz transform.

\begin{lemma}\label{splittoweak}
Let $r,r_1,r_2$ be such that $1<r_1<r<r_2<\infty$.
Then
\begin{align}
\|f\chi_{\{f\geq 1\}}\|_{L^{r_1}(\mathbb{R}^n)} \leq C(r,r_1)\|f\|^{\frac{r}{r_1}}_{L^{r,\infty}(\mathbb{R}^n)}\notag\\
\|f\chi_{\{f< 1\}}\|_{L^{r_2}(\mathbb{R}^n)} \leq C(r,r_2)\|f\|^{\frac{r}{r_2}}_{L^{r,\infty}(\mathbb{R}^n)}.\notag
\end{align}
\end{lemma}

\begin{proof}
Using the definition of weak $L^p$ spaces,
\begin{align}
\int_{\mathbb{R}^n} f^{r_1}\chi_{\{f\geq 1\}}&\leq \sum_{k=1}^\infty (2^{(k+1)r_1}-2^{kr_1} )\mu\{x|f(x)>2^k\}\notag\\
&\leq \sum_{k=1}^\infty \left( 2^{kr}\mu\{x|f(x)>2^k\}\right)2^{(r_1-r)k}(2^{r_1}-1)\notag\\
&\leq \|f\|_{L^{r,\infty}(\mathbb{R}^3)}\sum_{k=1}^\infty 2^{(r_1-r)k}(2^{r_1}-1).\notag
\end{align}
The sum on the right hand side is finite when $r_1<r$ and this proves the first statement.  The second statement follows in a similar way.
\end{proof}







\subsection{Local Energy Flux Estimates}
Following  \cite{MR2374209}, introduce the following scheme to localize energy.
\begin{align}
T_k &= -\frac{1}{2}\left(1+\frac{1}{2^k}\right)\notag\\
B_k&=B(1/2(1+2^{-3k}))\notag\\
Q_k &= [T_k,1]\times B_k \notag\\
v_k &= \{|u|-(1-\frac{1}{2^k})\}_+\notag\\
d_k^2&= v\frac{|\nabla u|^2}{|u|}+\chi_{v>0}\frac{(1-\frac{1}{2^k})}{|u|}|\nabla|u||^2\notag\\
U_k &=\frac{1}{2}\|v_k\|^2_{L^\infty((T_k,1);L^2(B_k))}+\int_{Q_{k}}d^2_k\,dx\,dt\notag
\end{align}

We also make use of the following inequality which measures how energy is transferred from one level set $U_k$ to the next.
\begin{proposition}\label{vasseurprop}(Vasseur)
Let $p>1$.  There exists universal constants $C_p,\beta_p>1$, depending only on $p$ such that for any solution to (\ref{NSE}), (\ref{genergy}) in $[-1,1]\times B(1)$, if $U_0\leq 1$ then we have for every $k>0$:
\begin{align}
U_k\leq C_p^k(1+\|P\|_{L^p(0,1;L^1(B_0))})U_{k-1}^{\beta_p}.\notag
\end{align}
\end{proposition}

\begin{proof}
See \cite{MR2374209}.
\end{proof}

\begin{remark}
In \cite{MR2374209} it was shown that this relation can be used to recover the partial regularity of Caffarelli, Kohn, and Nirenberg \cite{MR673830}.  Here we take a different approach and bound $U_0$ in terms of weak norms which we can scale small.
\end{remark}

\section{Local Study}
The goal of this section is to give universal control of $\|u(t)\|_\infty$ in terms of $\|u\|_{L^{\rho}((0,t);L^{\sigma,\infty}(\mathbb{R}^3))}$.  This relation is established in Proposition \ref{ineqprop} and will later be combined with a Gronwall argument to establish regularity.  To accomplish this we first establish uniform local control on $|u|$  (Proposition \ref{uniformboundprop}) then apply a scaling argument to obtain the global bound.

\subsection{Initial Energy and Pressure Bounds}
We begin by  establishing a way to bound the pressure term through the Riesz transform.
\begin{proposition}\label{pressureprop}
 Let $\rho,\sigma\in (3,\infty)$, and $r,p$ satisfy $1\leq r<\frac{\sigma}{2}$ and $1\leq p<\frac{\rho}{2}$.
Then, for any $\epsilon\in(0,1)$ there exists a $C(\epsilon)>0$ such that for any solution to (\ref{NSE}), (\ref{genergy}) in $\mathbb{R}^3\times [-2,1]$ satisfies
\begin{align}
\|P\|_{L^p((-2,1);L^r(B_{-1}))}&\leq C(\epsilon)\|u\|^{\frac{\rho}{2p}}_{L^{\rho}((-2,1);L^{\sigma,\infty}(\mathbb{R}^3))}+\epsilon\notag
\end{align}
Here the constant $C(\epsilon)$ also depends on $\rho,\sigma,r$ and $p$.
\end{proposition}
\begin{proof}

Let $R_i$ denote the Riesz transform, i.e. the integral operator with kernel $K_i(x)=c_n(x_i/|x|^{n+1})$.  The Riesz transform is a bounded linear operator mapping $L^p(\mathbb{R}^3)\rightarrow L^p(\mathbb{R}^3)$ for all $p\in (1,\infty)$ (see \cite{MR0304972}).  By taking the divergence of the Navier-Stokes equation one can see that the pressure can be represented as $P=\sum_{i,j}R_iR_j(u_iu_j)$.  We further decompose the pressure
\begin{align}
P&=P_1+P_2\notag\\
P_1&=\sum_{ij}R_iR_j(u_iu_j\chi_{\{|u|\geq 1\}})\notag\\
P_2&=\sum_{ij}R_iR_j(u_iu_j\chi_{\{|u|< 1\}}).\notag
\end{align}
Choose $r_1$ and $r_2$ such that $r<r_1<\frac{\sigma}{2}<r_2$ and $p<\frac{r_1\rho}{\sigma}$.  Using the boundedness of the Riesz transform and Lemma \ref{splittoweak} we bound
\begin{align}
\|P_1\|_{L^r(B_{-1})} &\leq C\|P_1\|_{L^{r_1}(B_{-1})}\notag\\
&\leq C\||u|^2\chi_{\{|u|\geq 1\}}\|_{L^{r_1}(\mathbb{R}^3)}\notag\\
&\leq C\||u|^2\|^{\frac{\sigma}{2r_1}}_{L^{\frac{\sigma}{2},\infty}(\mathbb{R}^3)}.\notag
\end{align}
The first inequality above relies on the fact that $B_{-1}$ is a compact set.  Similarly,
\begin{align}
\|P_2\|_{L^r(B_{-1})} &\leq C\||u|^2\|^{\frac{\sigma}{2r_2}}_{L^{\frac{\sigma}{2},\infty}}.\notag
\end{align}
Note $p<\frac{r_1\rho}{\sigma}<\frac{r_2\rho}{\sigma}$.  Using the two above estimates and relying on the compact time interval yields
\begin{align}
\|P_1\|_{L^{p}((-2,1);L^r(B_{-1}))} &\leq C(\epsilon)\||u|^2\|^{\frac{\rho}{2p}}_{L^{\frac{\rho}{2}}((-2,1);L^{\frac{\sigma}{2},\infty}(\mathbb{R}^3))}+\epsilon\notag\\
\|P_2\|_{L^{p}((-2,1);L^r(B_{-1}))} &\leq C(\epsilon)\||u|^2\|^{\frac{\rho}{2p}}_{L^{\frac{\rho}{2}}((-2,1);L^{\frac{\sigma}{2},\infty}(\mathbb{R}^3))}+\epsilon.\notag
\end{align}
The proposition follows immediately from these estimates.
\end{proof}

Our next step is to bound $U_0$ in terms of $\|u\|_{L^\rho((0,t);L^{\sigma,\infty}(\mathbb{R}^3))}$.
\begin{proposition}\label{initialboundprop}
Let $\sigma,\rho\in(3,\infty)$.  Given any $\epsilon\in(0,1)$ there exists a $C(\epsilon)>0$ such that any solution to (\ref{NSE}), (\ref{genergy}) in $\mathbb{R}^3\times [-2,1]$ satisfies
\begin{align}
U_0\leq C(\epsilon)\|u\|^\rho_{L^{\rho}((-2,1);L^{\sigma,\infty}(\mathbb{R}^3))} +\epsilon. \notag
\end{align}
\end{proposition}
\begin{proof}
Let $\eta$ be a smooth cutoff function satisfying
\begin{align}
\eta &= 1 \ \ \forall (x,t)\in Q_{0}\notag\\
\eta &=0 \ \ \forall (x,t)\in Q_{-1}^C\notag\\
\eta &\leq 1.\notag
\end{align}
Then, appealing to (\ref{genergy}), for $t>-1$ we have
\begin{align}
\int_{\mathbb{R}^3}u(t)^2\eta(t) \,dx +\int_{-2}^t\int_{\mathbb{R}^3}|\nabla u|^2\eta \,dx\,dt &\leq \int_{-2}^t\int_{\mathbb{R}^3}\frac{u^2}{2}(\eta_t+\triangle \eta) \,dx\,dt  \notag\\
&\ \ \ \ \ -\int_{-2}^t\int_{\mathbb{R}^3}(\nabla\eta\cdot u)\{\frac{u^2}{2}+P\} \,dx\,dt.\notag
\end{align}
The supremeum over $t\in (-2,1)$ of the left hand side above is greater then $U_0$ so it remains to bound the right hand side in terms of $\|u\|_{L^\rho((-2,1);L^{\sigma,\infty}(\mathbb{R}^3))}$.  Relying on the fact that $\eta$ is smooth and compactly supported, and appealing to Lemma \ref{weaklemma} we observe
\begin{align}
|\int_{-2}^t\int_{\mathbb{R}^3}(\nabla\eta\cdot u)\{\frac{u^2}{2}\} \,dx\,dt|&\leq C\|u\|^3_{L^3((-2,1);L^3(B_{-1}))}\notag\\
&\leq C(\epsilon)\|u\|^\rho_{L^\rho((-2,1);L^{\sigma,\infty}(B_{-1}))} +\epsilon^3.\notag
\end{align}
Similarly,
\begin{align}
|\int_{-2}^t\int_{\mathbb{R}^3}\frac{u^2}{2}(\eta_t+\triangle \eta) \,dx\,dt|\leq C(\epsilon)\|u\|^\rho_{L^\rho((-2,1);L^{\sigma,\infty}(B_{-1}))} +\epsilon^2.\notag
\end{align}

Lemma \ref{weaklemma}, and Proposition \ref{pressureprop} we find
\begin{align}
|\int_{-2}^t\int_{\mathbb{R}^3}(\nabla\eta\cdot u)P \,dx\,dt| &\leq  \|u\|_{L^{3}((-2,1);L^{3}(B_{-1}))} \|P\|_{L^{\frac{3}{2}}((-2,1);L^{\frac{3}{2}}(B_{-1}))}\notag\\
&\leq \left(C(\epsilon)\|u\|^{\frac{\rho}{3}}_{L^{\rho}((-2,1);L^{\sigma,\infty}(B_{-1}))} +\epsilon\right)\notag\\
&\ \ \ \ \ \ \ \cdot\left(C(\epsilon)\|u\|^{\frac{2\rho}{3}}_{L^{\rho}((-2,1);L^{\sigma,\infty}(\mathbb{R}^3))} +\epsilon\right)\notag\\
&\leq C(\epsilon)\|u\|^{\rho}_{L^{\rho}((-2,1);L^{\sigma,\infty}(\mathbb{R}^3))} +\epsilon^2+\epsilon^3.\notag
\end{align}
The proposition then follows immediately from these estimates.
\end{proof}


\subsection{Control of $\|u\|_\infty$}
We now establish the local control on $|u|$.
\begin{proposition}\label{uniformboundprop}
Let $\sigma,\rho\in(3,\infty)$.  There exists an absolute constant $C^*$, dependent only on $\sigma$ and $\rho$ such that for any solution to (\ref{NSE}), (\ref{genergy}) in $\mathbb{R}^3\times [-2,1]$ satisfying
\begin{align}
 \|u\|^\rho_{L^{\rho}((-2,1);L^{\sigma,\infty}(\mathbb{R}^3))} \leq C^*\notag
\end{align}
we have
$|u|<1$ on $[-1/2,1]\times \mathbb{R}^3$.
\end{proposition}

This proposition is a combination of Propositions \ref{vasseurprop}, \ref{initialboundprop}, and \ref{pressureprop} with a recursive lemma.  Before proceeding with the proof we recall the recursive lemma.
\begin{lemma}\label{recursivelemma}
For $C>1$ and $\beta>1$ there exists a constant $C_0^*<1$ such that for every sequence verifying $0<W_0<C_0^*$ and for every $k$:
\begin{align}
0\leq W_{k+1}\leq C^kW_k^\beta\notag
\end{align}
we have
\begin{align}
\lim_{k\rightarrow \infty}W_k=0.\notag
\end{align}
\end{lemma}
\begin{proof}
 See \cite{MR2374209}.
\end{proof}

Now we proceed with the proof of Proposition \ref{uniformboundprop}.
\begin{proof}
Let $C_0^*$ be as in \ref{recursivelemma}.  In Proposition \ref{initialboundprop} choose $\epsilon=\frac{C_0^*}{4}$ and let $C_{1}^*=C(\frac{C_0^*}{4})$ be the corresponding constant.  Set $C^*=\frac{C_{1}^*}{4}$.  Then if
\begin{align}
 \|u\|^\rho_{L^{\rho}((-2,1);L^{\sigma,\infty}(\mathbb{R}^3))} \leq C^*\notag
\end{align}
Proposition \ref{initialboundprop} yields $U_0<C_0^*<1$.  This, with Propositions \ref{vasseurprop}, \ref{initialboundprop} imply
\begin{align}
 U_k\leq C^kU_{k-1}^{\beta}\notag
\end{align}
where $C,\beta>1$.  It then follows from Lemma \ref{recursivelemma} that $\lim_{k\rightarrow\infty}U_k=0$.  Noticing that the PDE is invariant under spacial translations establishes the theorem.
\end{proof}

To complete this section we apply a scaling argument to bound $\|u\|_\infty$ in terms of the desired norms.  In the proposition below, it should be noted that the norms on both sides of the inequalities scale the same under the natural scaling of the Navier-Stokes equation, this is a consequence of the scaling argument.  In the following proposition $\rho$ and $\sigma$ must satisfy the relationship $\frac{2}{\rho}+\frac{3}{\sigma}<1$ which is less then the Lady{\v{z}}enskaja-Prodi-Serrin relation.  In the next section we shall pick a specific value of $\rho$ and $\sigma$ depending on $p$ and $q$ which will instead satisfy the usual Lady{\v{z}}enskaja-Prodi-Serrin relation.

\begin{proposition}\label{ineqprop}
Let $\sigma,\rho\in(3,\infty)$ be such that $\frac{2}{\rho}+\frac{3}{\sigma}<1$.  For any $\lambda \in (0,3)$ there exists an $A_\lambda>0$ such that if $u$ is a solution to (\ref{NSE}), (\ref{genergy}) then
\begin{align}
\|u(t)\|_{L^\infty(\mathbb{R}^3)} \leq A_\lambda \left(1+\|u\|^\frac{1}{{1-\frac{2}{\rho}-\frac{3}{\sigma}}}_{L^{\rho}((0,t);L^{\sigma,\infty}(\mathbb{R}^3))}\right)\notag
\end{align}
for all $t>\lambda$.
\end{proposition}
\begin{proof}
This proof is modeled on the proof of Proposition 1.1 in \cite{MR2437103}.  Let $C^*$ be as in Proposition \ref{uniformboundprop}. First we prove this proposition for $\lambda=3$.  Assume that $\|u\|^\rho_{L^\rho((0,t);L^(\sigma,\infty)(\mathbb{R}^3))} \leq C^*$.  Let $T>3$ be chosen arbitrarily.  By translating in time we can consider a new solution, $u(x,t+(T-1))$, which solves (\ref{NSE}), (\ref{genergy}) on the time interval $(-2,1)$.  From Proposition \ref{uniformboundprop} we conclude $\|u(T)\|_{L^\infty(\mathbb{R}^3)}\leq 1$.  If $\|u\|^\rho_{L^\rho((0,t);L^{\sigma,\infty}(\mathbb{R}^3)} > C^*$ consider the scaled solution $u_\epsilon(x,t)=u(\epsilon x,\epsilon^2 t)$ where
\begin{align}
\epsilon = \left(\frac{C^*}{\|u\|^{\rho}_{L^{\rho}((0,t/\epsilon^2);L^{\sigma,\infty}(\mathbb{R}^3))} }\right)^{\frac{1}{\rho(1-\frac{2}{\rho}-\frac{3}{\sigma})}}.\notag
\end{align}
This solutions satisfies $\|u_\epsilon\|^{\rho}_{L^{\rho}((0,t/\epsilon^2);L^{\sigma,\infty}(\mathbb{R}^3)} \leq C^*$ and we conclude, relying on the above argument, that $\|u_\epsilon\|_{L^\infty((0,T/\epsilon^2)\times \mathbb{R}^3} \leq 1$.  Noting
\begin{align}
\|u_\epsilon(T/\epsilon^2)\|_{L^\infty(\mathbb{R}^3)} = \epsilon \|u(t)\|_{L^\infty(\mathbb{R}^3)}\notag
\end{align}
the proof of this proposition is complete in the case $\lambda=3$.

If $\lambda \in (0,3)$ consider the scaled solution $w(x,t)=u_\epsilon(x,t)$ with $\epsilon=\left(\frac{\lambda}{3}\right)^\frac{1}{2}$.  Applying the above argument to $w$ yields the conclusion with $A_\lambda = \left(\frac{3}{\lambda}\right)^{\frac{1}{2}}A_3$.
\end{proof}

\section{Regularity Theorem}
We now apply a log Gronwall argument to Proposition \ref{ineqprop} and establish our regularity results.  Throughout this section let $p,q\in (3,\infty)$ be Prodi-Serrin exponents, i.e. $1=\frac{2}{p}+\frac{3}{q}$.  Also set
\begin{align}
\sigma = \frac{3(q-1)}{2} \ \ \ \ \ \ \ \ \rho=\frac{3(q-1)}{q-3}.\notag
\end{align}
The choice of $\sigma$ and $\rho$ is made for the following reasons.  First, with this choice $\|u\|^\rho_{L^\rho((0,t);L^\sigma(\mathbb{R}^3))}$ scales like $\|u\|_{L^\infty(\mathbb{R}^3\times (0,t))}$ under the natural scaling of the Navier-Stokes equation.  This is exactly the RHS of the inequality in Proposition \ref{ineqprop} since with this choice $1-\frac{2}{\rho}-\frac{3}{\sigma}=\frac{1}{\rho}$.  Second,
the H\"older inequality implies
\begin{align}
\|u\|^\rho_{L^{\rho}((0,t);L^{\sigma,\infty}(\mathbb{R}^3))} \leq \int_0^t \|u(s)\|_{L^\infty(\mathbb{R}^3)}\|u(s)\|^p_{L^{q,\infty}(\mathbb{R}^3)}\,ds.\label{holderrel2}
\end{align}
Using this relation with the Gronwall inequality we will establish our regularity result.

We work in the following situation
\begin{assumption}\label{ass1}
Let $u,P$ be a solution to (\ref{NSE}) satisfying (\ref{genergy}) in the sense of distributions.  In addition assume there is a $T^*>0$ such that $u\in L^\infty(\mathbb{R}^3\times[0,T^*])$.
\end{assumption}
\begin{remark}\label{remarkass}
The assumption is introduced because we are not focusing on local in time regularity.  Since it is known that Leray-Hopf solutions immediately become smooth, for any Leray-Hopf solution we can consider a new solution $\tilde{u}(x,t)=u(x,t+\epsilon)$ for $\epsilon$ small which satisfies this assumption.  This allows us to recover the theorem stated in the introduction from the one proven below.
\end{remark}

\begin{theorem}\label{logimprovethrm}
Let $u$ be a solution to (\ref{NSE}), (\ref{genergy}) satisfying Assumption \ref{ass1}.  If
\begin{align}
\int_0^t\frac{\|u(s)\|^p_{L^{q,\infty}(\mathbb{R}^3)}}{e+\log(e+\|u(s)\|_{L^\infty(\mathbb{R}^3)})}\,ds < \infty\notag
\end{align}
then $\|u\|_{L^\infty(\mathbb{R}^3\times (0,t))}<\infty$.
\end{theorem}
\begin{proof}
Fix $\lambda \in (0,T^*)$ where $T^*$ is as in Assumption \ref{ass1}.  Combining Proposition \ref{ineqprop} with (\ref{holderrel2}) shows, for all $t>\lambda$
\begin{align}
\|u(t)\|_{L^\infty(\mathbb{R}^3)} &\leq A_\lambda\left(1+\int_0^t \|u(s)\|_{L^\infty(\mathbb{R}^3)}\|u(s)\|^p_{L^{q,\infty}(\mathbb{R}^3)}\,ds\right)\notag\\
&\leq A_\lambda\left(1+\int_0^t \Psi(\|u(s)\|_{L^\infty(\mathbb{R}^3)}) B(s)\,ds\right) \notag\\
&\leq A_\lambda+A_\lambda\int_\lambda^t \Psi(\|u(s)\|_{L^\infty(\mathbb{R}^3)}) B(s)\,ds \notag\\
&\ \ \ \ \ \ \ \ +A_\lambda\sup_{t\in [0,\lambda]}\{\Psi(\|u(s)\|_{L^\infty(\mathbb{R}^3)})\}\int_0^\lambda  B(s)\,ds\notag\\
B(s)&=\frac{\|u(s)\|^p_{L^{q,\infty}(\mathbb{R}^3))}}{e+\log\{e+\|u(s)\|_{L^\infty(\mathbb{R}^3)}\}}\notag\\
\Psi(r)&= r(e+\log(e+r)).\notag
\end{align}
Relying on Assumption \ref{ass1} we note $\sup_{t\in [0,\lambda]}\{\Psi(\|u(s)\|_{L^\infty(\mathbb{R}^3)})\}<\infty$.  Since $\int B<\infty$ by assumption we write
\begin{align}
 \|u(t)\|_{L^\infty(\mathbb{R}^3)} &\leq C_\lambda\left(1+\int_\lambda^t \Psi(\|u(s)\|_{L^\infty(\mathbb{R}^3)}) B(s)\,ds\right).\label{RHSestimate}
\end{align}
Here the constant relies on $\lambda$ and the bound on $u$ assumed by Assumption \ref{ass1}.

The function $\Psi$ satisfies $\int_1^\infty\frac{1}{\Psi(r)}\,dr=\infty$ since $\frac{1}{\Psi(r)}>\frac{d}{dr}\log(e+\log(e+r))$.  This is enough to guarantee the result in the theorem.  Indeed, denote the RHS of (\ref{RHSestimate}) by $H(t)$ so that $\|u(t)\|_{L^\infty(\mathbb{R}^3)}\leq H(t)$.  Since $\Psi$ is monotone we compute
\begin{align}
\frac{d}{dt}H(t)= C\Psi(\|u(t)\|_{L^\infty(\mathbb{R}^3)}) B(t) \leq \Psi(H(t)) B(t).\notag
\end{align}
Integrating in time from $\lambda$ to $t$ yields
\begin{align}
\int_{C_\lambda}^{H(t)}\frac{1}{\Psi(s)}\,ds\leq C \int_\lambda^tB(s)\,ds.\notag
\end{align}
This holds for all $t>\lambda$.  Recall by assumption  $\int_0^tB(s)\,ds <\infty$ so the same is true for the left hand side.  Recalling $\int_1^\infty\frac{1}{\Psi(r)}\,dr=\infty$ we conclude $H(t)<\infty$, hence $\|u(t)\|_{L^\infty(\mathbb{R}^3)}<\infty$.
\end{proof}
Taking into account Remark \ref{remarkass} this theorem implies the one stated in the introduction.

\section{Appendix}
Throughout the appendix let $p$, $q$ satisify the condition $\frac{2}{p}+\frac{3}{q}=1$.  We now construct a function $f(x,t)$ for which 
\begin{align}
\int_0^T\frac{\|f(s)\|^p_{L^{q,\infty}(\mathbb{R}^3)}}{e+\log(e+\|f(s)\|_{L^\infty(\mathbb{R}^3)})}\,ds < \infty\label{bound}
\end{align}
but $f\notin L^{p,r}((0,T);L^{q,\infty}(\mathbb{R}^3))$ for any $r<\infty$.  First recall the norm
\begin{align}
\|f\|_{L^{p,r}(\Omega)}= \left(\int_0^\infty R^{r-1}[\lambda_{f,\Omega}(R)]^{\frac{r}{p}}\,dR\right)^{\frac{1}{r}}.\notag
\end{align}

Let
\begin{align}
f(x,t)=\frac{A(t)}{\sqrt{(t_\infty-t)+(x-x_0)^2}}\notag
\end{align}
with $A(t)\geq 0$ to be defined shortly.  By direct calculation one finds
\begin{align}
\|f\|_{L^{q,\infty}(\mathbb{R}^3)}&=\frac{A(t)}{(t_\infty-t)^{\frac{1}{p}}},\notag\\
\|f\|_{L^\infty(\mathbb{R}^3)}&=\frac{A(t)}{(t_\infty-t)^{\frac{1}{2}}}.\notag
\end{align}

Define $A$ in the following way:
\begin{align}
A(t)=\sum_{n=1}^\infty 2^{m_n}&\chi_{(t_n, t^*_n)}(t),\notag\\
m_n=n^{2}-\frac{n}{2}\ \ \ \ \ \ &\ \ \ \ \  k_n=pm_n+n,\notag\\
t_n= t_\infty(1-2^{-n})\ \ \ \ \ &\ \ \ \ \ t_n^*= t_n+t_\infty2^{-k_n}.\notag
\end{align}
Here, $\chi_I$ is the indicator function for the interval $I$.  Note that for $t\in(t_n, t^*_n)$, 
\begin{align}
t_\infty(2^{-n}-2^{-k_n})\leq (t_\infty-t)\leq t_\infty 2^{-n}.\notag
\end{align}

{\bf Claim 1:}
\begin{align}
\int_0^{t_\infty}\frac{\|f(s)\|^p_{L^{q,\infty}(\mathbb{R}^3)}}{e+\log(e+\|f(s)\|_{L^\infty(\mathbb{R}^3)})}\,ds < \infty\notag
\end{align}

To prove this claim evaluate directly.  The integral in question is bounded by
\begin{align}
t_\infty^{-1}\sum_{n=1}^\infty \frac{2^{-k_n}2^{pm_n}(2^{-n}-2^{-k})^{-1}}{e+\log(e+t_\infty^{-\frac{1}{2}}2^{m_n+\frac{n}{2}})}
\leq C\sum_{n=1}^\infty \frac{1}{n^{2}}.\notag
\end{align}
This establishes the claim

{\bf Claim 2:}
For any $r\in (1,\infty)$,
\begin{align}
\|f\|_{L^{p,r}((0,t_\infty);L^{q,\infty}(\mathbb{R}^3))} = \infty.\notag
\end{align}
Set $R_n=2^{\frac{k_n}{p}}t_\infty^{\frac{1}{p}}$.  Then,
\begin{align}
\mu(\{(t_\infty-t)^{-\frac{1}{p}}A(t)>R\})\geq\sum_{n=n_0}^\infty2^{-k_n}\geq t_\infty R^{-p}_{n_0}\ \ \ \ \ \ if \ \ \ R<R_{n_0}.\notag
\end{align}
This implies
\begin{align}
\|f\|^r_{L^{p,r}((0,t_\infty);L^{q,\infty}(\mathbb{R}^3))}\geq t_\infty^{r/p}\sum_{n=1}^\infty \int_{R_{n-1}}^{R_n} R^{r-1}R_n^{-r}dR\notag\\
=t_\infty^{r/p}\sum_{n=1}^\infty R_n^{-r}(R_n^{r}-R_{n-1}^r).\notag
\end{align}
If we can show that the sum
\begin{align}
\sum_{n=1}^\infty R^r_{n-1}R^{-r}_n\notag
\end{align}
is convergent then we have proven the claim.  This sum is exactly
\begin{align}
\sum_{n=1}^\infty (2^{k_{n-1}-k_n})^{\frac{r}{p}}=2^{\frac{3r}{2}-\frac{r}{p}}\sum_{n=1}^\infty 2^{-2nr}\notag
\end{align}
which converges for all choices of $r>1$.

\bibliographystyle{plain}

\bibliography{weaklogimprove_local}

\end{document}